\documentclass[]{article}
\usepackage[sortcites]{biblatex}
\usepackage[]{amsmath}
\usepackage[]{amsthm}
\usepackage[]{amsfonts}
\usepackage[]{geometry}
\usepackage[]{tcolorbox}
\usepackage[]{hyperref}

\newtheorem{lemma}{Lemma}[section]
\newtheorem{theorem}{Theorem}[section]
\newtheorem{proposition}{Proposition}[section]
\theoremstyle{definition}
\newtheorem{definition}{Definition}

\DeclareMathOperator{\res}{res}
\DeclareMathOperator{\lc}{lc}
\newcommand{\tri}[3]{\ensuremath{#1^{#2} - #1^{#3} + 1}}

\title{Dyadically resolving trinomials for fast modular arithmetic}
\author{Robert Dougherty-Bliss \and Mits Kobayashi \and Natalya Ter-Saakov \and Eugene Zima}
\date{\today}

\begin{document}

\maketitle

\begin{abstract}
    Residue number systems based on pairwise relatively prime moduli are a
    powerful tool for accelerating integer computations via the Chinese
    Remainder Theorem. We study a structured family of moduli of the form $2^n -
    2^k + 1$, originally proposed for their efficient arithmetic and bit-level
    properties. These trinomial moduli support fast modular operations and
    exhibit scalable modular inverses.

    We investigate the problem of constructing large sets of pairwise
    relatively prime trinomial moduli of fixed bit length. By analyzing the
    corresponding trinomials $x^n - x^k + 1$, we establish a sufficient
    condition for coprimality based on polynomial resultants. This leads to a
    graph-theoretic model where maximal sets correspond to cliques in a
    compatibility graph, and we use maximum clique-finding algorithms to
    construct large examples in practice. Based on the theory of graph
    colorings, polynomial resultants, and properties of cyclotomic polynomials,
    we also prove upper bounds on the size of such sets as a function of $n$.
\end{abstract}

\section{Introduction}

A popular technique to speed up computations with integer arithmetic is to
reduce the input modulo several relatively prime numbers, compute with the
residues, then reconstruct the result with the Chinese remainder theorem
\cite{knuth,doliskani_simultaneous_2018,zima_cunningham_2007}. In \cite{zima},
it was proposed to use ``trinomial'' moduli of the form
\begin{equation}
    \label{trinom2}
    2^n - 2^k + 1, \quad 0 < k < n.
\end{equation}
These numbers have nice binary representations which make reduction into a
modular system fast. Sometimes, they exhibit ``scalable inverses''. By that, we
mean that replacing $2$ with $2^c$ does not change the sparsity or bit-pattern
of the modular inverses. For example, the moduli $2^{20} - 2^{12} + 1$ and
$2^{20} - 2^4 + 1$ satisfy the following property: for any integer $c \geq 1$,
\begin{alignat*}{2}
    ((2^c)^{20}-(2^c)^{12}+1)^{-1} &\equiv (2^c)^{8} + 1 &&\pmod{(2^c)^{20}-(2^c)^{4}+1}\\
    ((2^c)^{20}-(2^c)^{4}+1)^{-1} &\equiv (2^c)^{20} - (2^c)^{12} - (2^c)^{8} + 1 &&\pmod{(2^c)^{20}-(2^c)^{12}+1}.
\end{alignat*}
This means that the reconstruction step of the Chinese Remainder Theorem can be
carried out quickly, even with arbitrarily large moduli. Using ad-hoc methods,
the authors of \cite{zima} discovered a set of five pairwise relatively prime
moduli of this form and used them in upper-layer on top of
\cite{doliskani_simultaneous_2018} to show improvement in a standard integer
matrix multiplication benchmark.

Their work left open the following questions:
\begin{enumerate}
    \item When are two trinomial moduli relatively prime with scalable inverses?
    \item Are there arbitrarily large sets of trinomial moduli with the same
    bit length which are pairwise relatively prime with scalable inverses?
    \item For a fixed bit length, how can we efficiently find these sets?
\end{enumerate}
Our aim is to answer some of these questions by examining the pure polynomial
trinomials
\begin{equation}
    \label{trinomx}
    x^n - x^k + 1, \quad 0 < k < n.
\end{equation}

It might seem reasonable to establish the relative primality of trinomial pairs
of the form \eqref{trinomx}, then use this to deduce the relative primality of
moduli of the form \eqref{trinom2}. Unfortunately, this does not work. Almost
all pairs of trinomials of the form \eqref{trinomx} are relatively prime over
the rationals, yet only a small proportion of moduli of the form \eqref{trinom2}
are relatively prime integers. In other words, the substitution $x = 2$ in
\eqref{trinomx} does not preserve relative primality.

Our main results are as follows:
\begin{enumerate}
    \item A simple condition based on the resultant of two trinomials which
    ensures the relative primality of moduli of the form and the scalability of
    their modular inverses.
    \item A proof that arbitrarily large sets of relatively prime moduli with scalable inverses exist.
    \item An upper bound on the number of such moduli for a fixed bit length.
\end{enumerate}

\section{Background and definitions}

\begin{definition}
    The \emph{resultant} of two polynomials $f(x)$ and $g(x)$ over $\mathbb{Q}$
    with roots $\alpha_i$ and $\beta_j$, respectively, is
    \begin{equation*}
        \res(f, g) = \lc(f)^{\deg g} \lc(g)^{\deg f} \prod_{\substack{1 \leq i \leq \deg f \\ 1 \leq j \leq \deg g}} (\alpha_i - \beta_j),
    \end{equation*}
    where $\lc(f)$ is the leading coefficient of $f$. The resultant satisfies
    the following properties:
    \begin{align*}
        \res(f, g) &= \lc(f)^{\deg g} \prod_{i = 1}^{\deg f} g(\alpha_i) \\
        \res(f, g) &= (-1)^{\deg f \deg g} \res(g, f) \\
        \res(fg, p) &= \res(f, p) \res(g, p), \quad p \in \mathbb{Q}[x].
    \end{align*}
    If $g = fq + r$, where $\deg r < \deg f$, then
    \begin{equation*}
        \res(f, g) = \lc(f)^{\deg g - \deg r} \res(f, r).
    \end{equation*}
\end{definition}

\begin{definition}
    A pair of polynomials with integer coefficients \emph{dyadically
    resolve} if their resultant is a signed power of 2, meaning it is of the form $\pm 2^k$ for
    some nonnegative integer $k$.
\end{definition}

This definition is motivated by \cite{zima}, where it was desired that numbers
of the form $2^{cn} - 2^{ck} + 1$ would be relatively prime with ``stable''
inverses for sufficiently large integers $c$. By that, we mean that there exists
a polynomial $f(x)$ such that
\begin{equation*}
    (2^{cn} - 2^{ck} + 1)^{-1} \equiv f(2^c) \pmod{2^{cn} - 2^{cj} + 1}
\end{equation*}
for all sufficiently large values of $c$. It turns out that this is equivalent
to stating that the polynomials $x^n - x^k + 1$ and $x^n - x^j + 1$ dyadically
resolve.

The basic idea is that the inverse polynomial $f$ should be a B\'ezout cofactor
of $x^n - x^k + 1$ and $x^n - x^j + 1$. Assuming that the trinomials are
relatively prime, we can write \begin{equation*} (x^n - x^k + 1) f(x) + (x^n -
x^j + 1) g(x) = 1 \end{equation*} for some polynomials $f$ and $g$ over
$\mathbb{Q}$. The resultant of two polynomials is the determinant of their
Sylvester matrix, so if the trinomials dyadically resolve then $f$ and $g$ can
be constructed using only dyadic coefficients. Then, plugging in $x = 2^c$ for
large enough $c$ produces a statement about modular inverses. Getting the
details exactly right and proving the converse---``scalable inverses'' imply
dyadic resolvability---is more complicated and handled in
Section~\ref{sec:equivalence}.

\begin{definition}
    For a fixed integer $n$, let $T(n)$ be the graph on the vertices $\{1, 2,
    \dots, n - 1\}$ which contains the edge $\{i, j\}$ if and only if
    $\tri{x}{n}{i}$ and $\tri{x}{n}{j}$ dyadically resolve. Let $\omega(T(n))$
    be the largest size of a clique in $T(n)$.
\end{definition}

Many important questions about trinomial moduli can be reformulated as questions
about the graphs $T(n)$. A set of pairwise dyadically resolving trinomials
corresponds to a clique (complete subgraph) in $T(n)$; the percent of pairs of
trinomials which dyadically resolve is the edge density of $T(n)$; and so on.
See Figure~\ref{T-40-100} for striking drawings of $T(n)$.

\begin{figure}[t]
    \centering
    \includegraphics[width=0.5\linewidth]{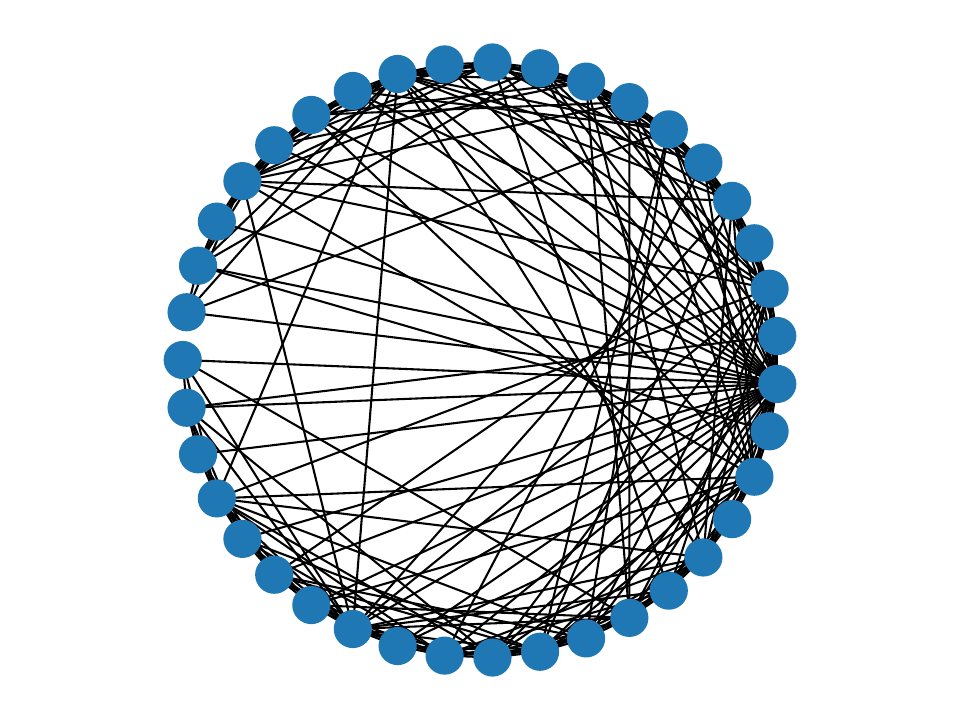}%
    \includegraphics[width=0.5\linewidth]{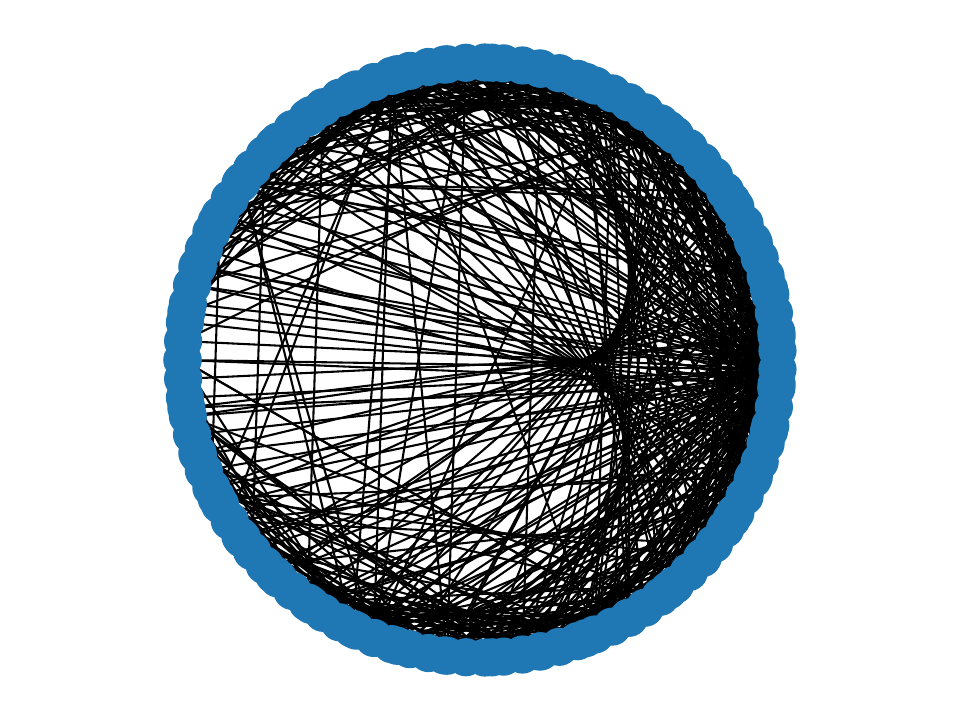}
    \caption{$T(40)$ and $T(100)$}
    \label{T-40-100}
\end{figure}

\begin{definition}
    $\nu_2(n)$ is the \emph{2-adic valuation of $n$}, meaning $\nu_2(n) = v$ is the
    largest $v$ such that $2^v$ divides $n$.
\end{definition}

\begin{definition}
    A \emph{root of unity} is a complex number $z$ such that $z^n = 1$ for some
    positive integer $n$. The smallest such $n$ is called the \emph{order} of
    $z$, and $z$ is also called a \emph{primitive} root of unity of its order.
    The primitive $n$th roots of unity all have the same minimal polynomial over
    $\mathbb{Q}$, the $n$th cyclotomic polynomial $\Phi_n(x)$. The cyclotomic
    polynomials corresponding to powers of $2$ have the simple formula
    \begin{equation*}
        \Phi_{2^t}(x) = x^{2^{t - 1}} + 1
    \end{equation*}
    for $t \geq 1$.
\end{definition}

\section{Computational results}

The main computational problem is to determine the largest possible set of
pairwise dyadically resolving trinomials of degree $n$, which is equivalent to
finding the the maximum cliques of $T(n)$. The graph theoretic formulation
gives us a computational foothold, because finding the maximum cliques in a
graph is known to take less time than the brute force approach. The
Bronn--Kerbosch algorithm runs in time $O(3^{n / 3})$ \cite{bron-kerbosch} and
Robson's algorithm runs in time $O(2^{n/4})$ \cite{robson}.

\paragraph{Clique numbers.} We have constructed $T(n)$ for $n \leq 3000$ and
determined all of their maximum cliques. These results are available at
\url{https://github.com/rdbliss/trinomials}. The largest known clique in
\cite{zima} had size five, but we now know several cliques of size ten. These
cliques were used in two-layer modular arithmetic package to further improve
performance reported in \cite{zima}. We do not know the first $n$ for which a
clique of size 11 occurs in $T(n)$.

Figure~\ref{cliques}, a plot of all known clique numbers, suggests that
$\omega(T(n))$ grows extremely slowly. In the following sections we will show
that
\begin{equation*}
    \lim_{n \to \infty} \omega(T(n)) = \infty
    \quad \text{and} \quad
    \omega(T(n)) = O(\log n),
\end{equation*}
but the true rate of growth is unknown. For example, our arguments imply that
$\omega(T(n)) \geq 6$ for $n > 48\ 439\ 664$, but Figure~\ref{cliques} suggests
that $\omega(T(n)) \geq 6$ for $n > 500$.

\begin{figure}
    \centering
    \includegraphics[width=\textwidth]{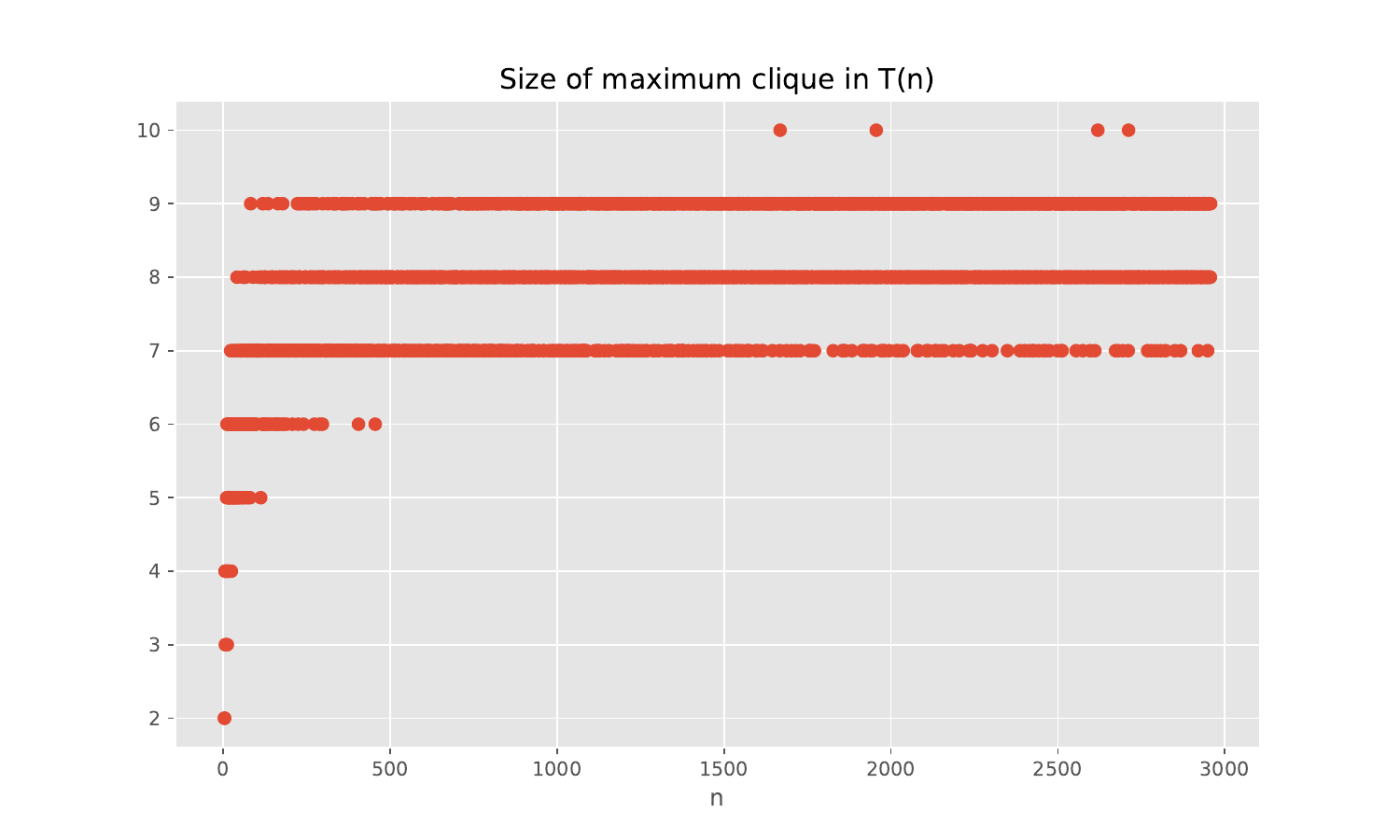}
    \caption{Clique number of $T(n)$ for $n \leq 3000$.}
    \label{cliques}
\end{figure}

\paragraph{Edge density.} We do not know of any simple, widely applicable
condition to predict the edges of $T(n)$. Some sufficient conditions are given
in the following section, but they explain a small fraction of the edges in
$T(n)$. The adjacency matrix of $T(200)$ in Figure~\ref{adj-T-200} has many
interesting features. Empirically, it seems that very few pairs of trinomials
dyadically resolve. We conjecture that the number of edges in $T(n)$ is
$o(n^2)$.

\begin{figure}[t]
    \centering
    \includegraphics[width=0.7\textwidth]{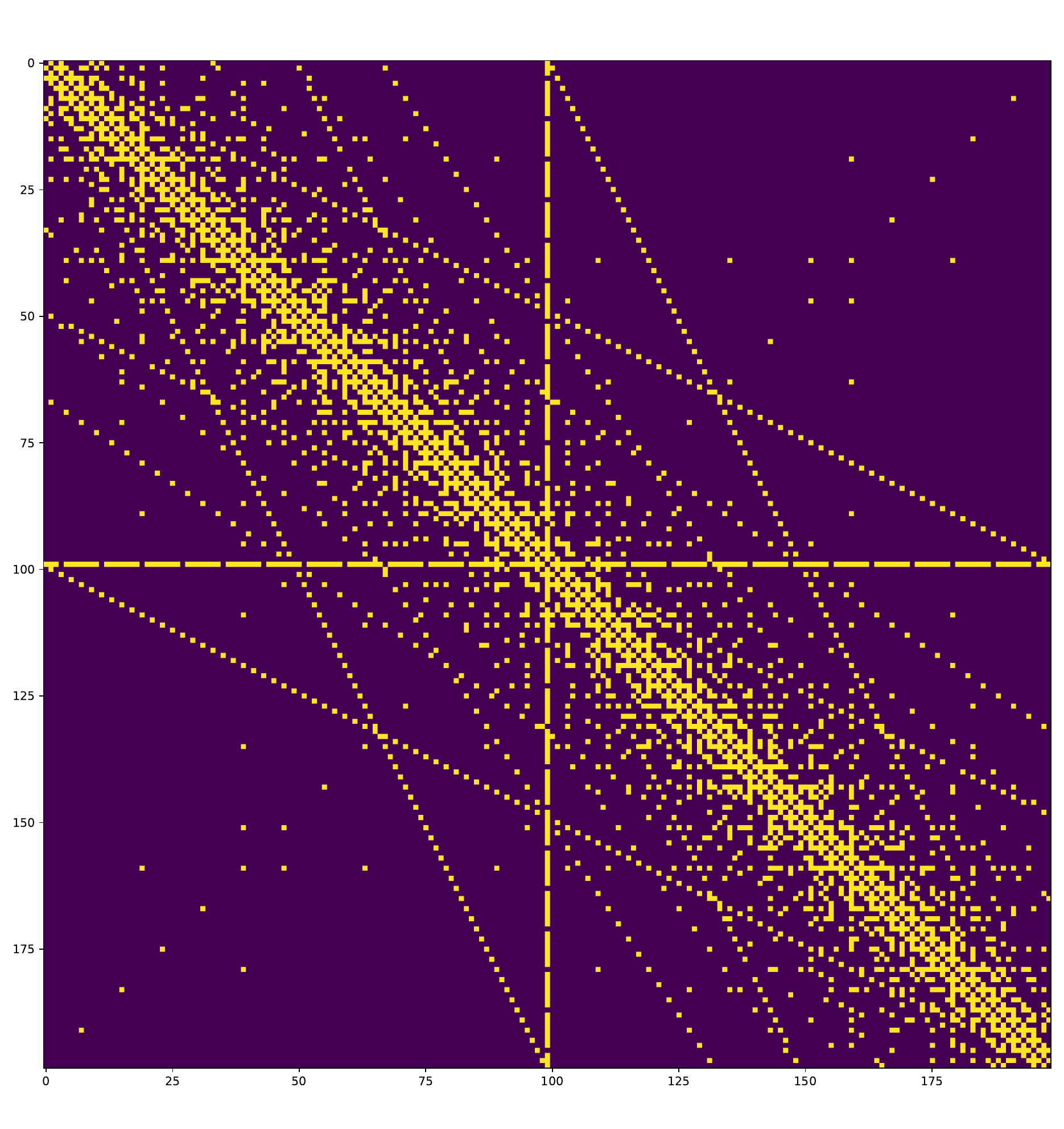}
    \caption{Adjacency matrix of $T(200)$. Edges are marked with yellow, and
    non-edges with purple.}
    \label{adj-T-200}
\end{figure}

\begin{figure}
    \centering
    \includegraphics[width=0.7\textwidth]{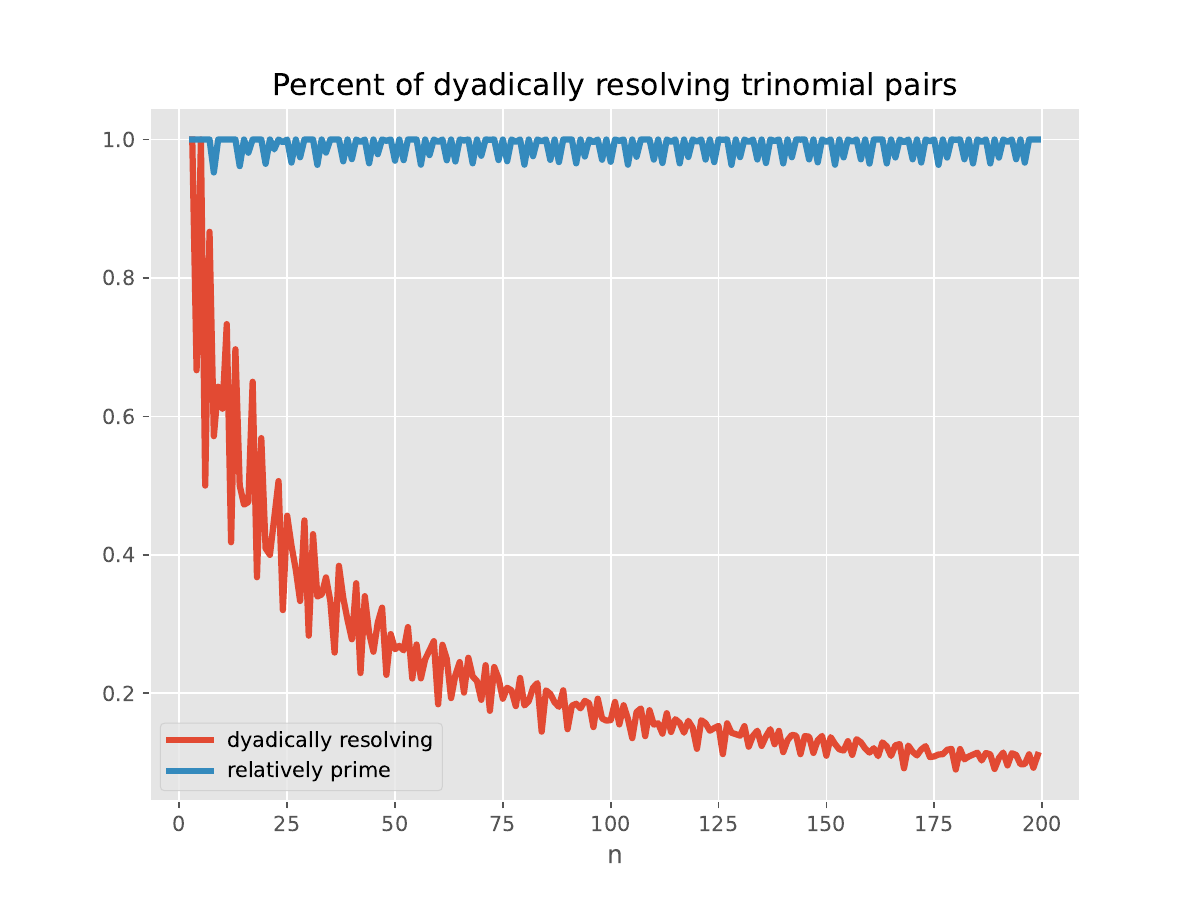}
    \caption{Percent of trinomial pairs which are relatively prime (blue) and
    dyadically resolving (red). Almost all pairs are relatively prime, but a
much smaller proportion dyadically resolve.}
    \label{percents}
\end{figure}

Finally, we mention that the main computational bottleneck seems to be
constructing $T(n)$, not finding cliques. It is possible to construct $T(n)$ for
$n$ as large as 3000 or a bit higher, but significantly larger $n$ would likely
require some clever computation.

\section{Properties of the trinomial graph}

The goal of this section is to prove that there are cliques of arbitrary sizes
in $T(n)$ as $n \to \infty$. This implies that the optimization problem of
finding the \emph{best} cliques is well-founded. We will also prove a number of
interesting properties unrelated to our main problem.

\begin{lemma}
    \label{simple-seq}
    \tri{x}{n}{k} and \tri{x}{n}{j} dyadically resolve if $k - j$ divides $k$.
\end{lemma}

\begin{proof}
    Without loss of generality, suppose that $k > j$. Then, the resultant of
    two such trinomials can be computed as follows:
    \begin{align*}
        \res(\tri{x}{n}{k}, \tri{x}{n}{j})
        &=
        \pm\res(\tri{x}{n}{k}, x^k - x^j) \\
        &=
        \pm\res(\tri{x}{n}{k}, x^j) \res(\tri{x}{n}{k}, x^{k - j} - 1).
    \end{align*}
    The first factor is $\pm 1$, because
    \begin{equation*}
        \res(\tri{x}{n}{k}, x^j) = \pm\prod_{w^j = 0} (\tri{w}{n}{k}) = \pm1.
    \end{equation*}
    For the second factor, the assumption that $k - j$ divides $k$ implies
    \begin{equation*}
        \res(\tri{x}{n}{k}, x^{k - j} - 1)
        =
        \pm \prod_{w^{k - j} = 1} (\tri{w}{n}{k})
        =
        \pm \prod_{w^{k - j} = 1} w^n.
    \end{equation*}
    This last product is
    \begin{equation*}
        \prod_{w^{k - j} = 1} w^n
        = \left( \prod_{w^{k - j} = 1} w =  \right)^n
        = \pm 1,
    \end{equation*}
    where we have used the fact that the product of all roots of unity of a
    given order is $\pm 1$.
\end{proof}

\begin{theorem}
    Sufficiently large trinomial graphs contain cliques of arbitrary size.
    That is,
    \begin{equation*}
        \lim_{n \to \infty} \omega(T(n)) = \infty.
    \end{equation*}
\end{theorem}

\begin{proof}
    Suppose that we have constructed a sequence of positive integers $k_1 < k_2
    < \cdots < k_m$ such that
    \begin{equation}
        \label{k-divis}
        k_i - k_j \quad \text{divides} \quad k_i,\quad i \neq j.
    \end{equation}
    Such a sequence would, by Lemma~\ref{simple-seq}, produce a clique in
    $T(n)$ for any $n > k_m$. Set $k = \operatorname{lcm}(k_1, k_2, \dots,
    k_m)$, and then define the new sequence
    \begin{equation*}
        k_i' = \frac{k}{k_i} (k_i + 1).
    \end{equation*}
    We claim that $k < k_1' < k_2' < \cdots < k_m'$ is a sequence of values
    that satisfies the divisibility condition \eqref{k-divis}, but is one term
    longer.

    To see why this is true, note that
    \begin{equation*}
        \frac{k_i'}{k_i' - k_j'} = \frac{(k_i + 1) k_j}{k_j - k_i}.
    \end{equation*}
    By assumption, $k_j - k_i$ divides $k_j$, so this quotient is an integer.
    It remains to check that $k_i' - k$ divides both $k_i'$ and $k$ for every
    $i$, but this is simple given the following equations:
    \begin{align*}
        \frac{k_i'}{k_i' - k} &= k_i + 1 \\
        \frac{k}{k_i' - k} &= k_i.
    \end{align*}
    This inductive process constructs such sequences of arbitrary length, which
    shows that every $T(n)$ for $n$ large enough will contain arbitrarily large
    cliques.
\end{proof}

To make the preceding theorem concrete, we can begin the inductive process with
$k_1 = 1$. That produces the following cliques:
\begin{align*}
    &\{1\} \\
    &\{1, 2\} \\
    &\{2, 3, 4\} \\
    &\{12, 15, 16, 18\} \\
    &\{720, 760, 765, 768, 780\} \\
    &\{48372480, 48434496, 48435465, 48435712, 48436128, 48439664\}.
\end{align*}
These sequences are theoretically interesting but grow too quickly to be useful
in practice. (It is even nontrivial to verify \emph{a posteriori} that the last
set is a clique in $T(48439665)$.) Nonetheless, they show that there \emph{are}
large cliques in $T(n)$.

Figure~\ref{T-40-100} suggests that $T(n)$ contains a cardioid, a circle, and
has some reflectional symmetry. The following proposition shows that these are
all true.

\begin{proposition}
    $T(n)$ contains a cardioid and (almost) a circle. That is, it contains the
    following edges:
    \begin{itemize}
        \item $\{k, 2k\}$ and $\{\lfloor n / 2 \rfloor + k + n \bmod 2, 2k + n \bmod 2\}$ for $1 \leq k
            \leq n / 2$.
        \item $\{k, k + 1\}$ for $1 \leq k < n - 1$
    \end{itemize}
    Further, $T(n)$ contains $\{i, j\}$ if and only if it contains $\{n - i, n
    - j\}$.
\end{proposition}

\begin{proof}
    The first two families follow from simple resultant computations. For
    example:
    \begin{align*}
        \res(x^n - x^k + 1, x^n - x^{2k} + 1)
        &= \pm \res(x^n - x^k + 1, x^k - 1) \\
        &= \pm \res(x^n, x^k - 1) \\
        &= \pm 1.
    \end{align*}
    The symmetry follows from noting that the resultant changes only in sign when the
    coefficients of a polynomial are reversed.
\end{proof}

\section{An upper bound on the clique number}

Let $a(k)$ be the smallest $n$ such that there exists a clique of size $k$ in
$T(n)$. The first few values of $a(k)$ are listed below.

\begin{center}
    \begin{tabular}{c | c} $k$ & $a(k)$ \\
        \hline
        2 & 3 \\
        3 & 5 \\
        4 & 5 \\
        5 & 10 \\
        6 & 11 \\
        7 & 22 \\
        8 & 41 \\
        9 & 82 \\
        10 & 1668 \\
        11 & 3000 $<$ ??? $< 10^{2774}$
    \end{tabular}
\end{center}
Computing $a(10) = 1668$ took long enough that we proved the following theorem
while we waited.

\begin{theorem}
    \label{clique-bound} The largest clique in $T(n)$ has size no more than $2
    \lfloor \log_2 n \rfloor - \nu_2(n)$.
\end{theorem}

We will prove this theorem with a graph coloring argument. An \emph{independent
set} in a graph is a collection of vertices which have no edges between them. A
\emph{coloring} of a graph is a partition of its vertices into independent sets,
called ``colors.'' In $T(n)$, an independent set is a collection of trinomials
which are pairwise non-dyadically resolving. If a graph can be colored with $k$
colors, then it contains no cliques of size larger than $k$.

The following lemma shows that independent sets in $T(n)$ can sometimes be
obtained from congruence classes.

\begin{proposition}
    Fix positive integers $i \leq d < n$. If $x^n - x^i + 1$ and $x^d - 1$
    dyadically resolve, then $T(n)$ contains the path $[i, i + d, i + 2d,
    \dots]$. Otherwise, $\{i, i + d, i + 2d, \dots\}$ is an independent set.
\end{proposition}

\begin{proof}
    If $k \equiv j \pmod{d}$, then a few Euclidean divisions show that
    \begin{equation*}
        \res(x^n - x^k + 1, x^n - x^j + 1)
        =
        \pm
        \res(x^n - x^k + 1, x^{k - j} - 1).
    \end{equation*}
    Because $d$ divides $k - j$, this resultant is divisible by
    \begin{equation*}
        \res(\tri{x}{n}{k}, x^d - 1).
    \end{equation*}
    This last resultant only depends on the value of $k \bmod d$, so we may
    replace $k$ with any member of the congruence class. If the congruence
    class $k \bmod d$ is not a path, then this final resultant is not a signed
    power of $2$ for some member of the class, from which it follows that the
    edge $\{k, j\}$ does not exist in $T(n)$.
\end{proof}

\begin{lemma}
    If either $x^n + 2$ or $2 x^n + 1$ dyadically resolve with $x^{2^j} + 1$,
    then $\nu_2(n) = j$.
\end{lemma}

\begin{proof}
    Richard Swan found a remarkable formula for the resultant of two binomials
    \cite{swan}. His formula implies
    \begin{equation*}
        \res(x^n + 2, x^{2^j} + 1) = \pm((-2)^{2^j / g} - (-1)^{n/g})^g
    \end{equation*}
    where $g = \gcd(2^j, n)$. This is a power of 2 if and only if
    \begin{equation*}
        |(-2)^{2^j / g} - (-1)^{n / g}| = 1,
    \end{equation*}
    which is equivalent to $2^j / g = 1$ and $\nu_2(n / g) = 0$. In other words,
    the resultant is a power of 2 if and only if $\nu_2(n) = j$. For the other
    case, note that $2x^n + 1$ is the reciprocal polynomial of $x^n + 2$, so the
    two resultants are the same up to a sign change.
\end{proof}

\begin{proposition}
    \label{missed-classes} If $n$ is odd, then all congruence classes modulo
    $2^j$ are independent sets in $T(n)$ except possibly $0 \bmod 2^j$ and $n
    \bmod 2^j$.
\end{proposition}

\begin{proof}
    We will show, by induction on $j$, that $\tri{x}{n}{k}$ and $x^{2^j} - 1$
    dyadically resolve only if $k \equiv 0 \pmod{2^j}$ or $k \equiv n
    \pmod{2^j}$. For the base case $j = 1$, we have
    \begin{equation*}
        \res(\tri{x}{n}{k}, x^2 - 1) = (-1)^n - (-1)^k + 1.
    \end{equation*}
    A simple check shows that this is a signed power of 2 exactly when $k
    \equiv n \pmod{2}$ or $k \equiv 0 \pmod{2}$.

    For the inductive step, fix $j \geq 2$ such that $\res(\tri{x}{n}{k},
    x^{2^j} - 1)$ is a signed power of 2. Then $\res(\tri{x}{n}{k}, x^{2^{j -
    1}} - 1)$ is also a signed power of 2---because the latter divides the former---so the induction hypothesis applies. It follows that $k$ is $0$ or $n$ mod
    $2^{j - 1}$, and so there are four possible values for $k \bmod 2^j$:
    \begin{equation*}
        0,\ 2^{j - 1},\ n,\ n + 2^{j - 1}.
    \end{equation*}
    If $k \equiv 2^{j - 1} \pmod{2^j}$, then
    \begin{equation*}
        \res(\tri{x}{n}{k}, \Phi_{2^j}(x))
        = \res(x^n + 2, x^{2^{j - 1}} + 1)
    \end{equation*}
    because $z^k = -1$ for any primitive $2^j$-th root of unity $z$. If this is
    a signed power of 2, then the previous lemma implies $\nu_2(n) = j - 1 \geq
    1$. But we know that $n$ is odd, so this cannot happen. The case $k
    \equiv n + 2^{j - 1} \pmod{2^j}$ is similarly dispatched.
\end{proof}

\begin{proposition}
    \label{independent-sets}
    If $n = 2^{\nu_2(n)} n_1$, then the following congruences classes are
    independent sets in $T(n)$:
    \begin{align*}
        &\{k \equiv 1 \pmod{2}\} \\
        &\{k \equiv 2 \pmod{4}\} \\
        &\{k \equiv 4 \pmod{8}\} \\
        &\cdots \\
        &\{k \equiv 2^{\nu_2(n) - 1} \pmod{2^{\nu_2(n)}}\}.
    \end{align*}
    The vertices not included in these sets are $\{2^{\nu_2(n)} j \mid 0 < j <
    n_1\}$. They form an isomorphic copy of $T(n_1)$.
\end{proposition}

\begin{proof}
    If $k \equiv j \equiv 2^{i - 1} \pmod{2^i}$ where $i \leq \nu_2(n)$, then
    \begin{equation*}
        \res(\tri{x}{n}{k}, \tri{x}{n}{j})
        =
        \res(\tri{x}{n}{k}, x^{k - j} - 1)
    \end{equation*}
    is divisible by
    \begin{equation*}
        \res(\tri{x}{n}{k}, \Phi_{2^i}(x)).
    \end{equation*}
    If $z$ is a root of unity of order $2^i$, then $z^n - z^k + 1 = 3$ because
    $z^k = -1$ and $z^n = 1$, so this resultant is $3^{\deg \Phi_{2^i}(x)} =
    3^{2^{i - 1}}$.

    The vertices not in any of these congruence classes are of the form
    $2^{\nu_2(n)} j$ for $0 < j < n_1$. To see that these form an isomorphic copy
    of $T(n_1)$, note that
    \begin{equation*}
        \res(f(x^a), g(x^a)) = \res(f(x), g(x))^a
    \end{equation*}
    for any positive integer $a$, meaning that $f(x^a)$ and $g(x^a)$ dyadically
    resolve if and only if $f(x)$ and $g(x)$ do. Thus, if $g = \gcd(n, k, j)$,
    then $\{k, j\}$ is an edge in $T(n)$ if and only if $\{k/g, j/g\}$ is an
    edge in $T(n/g)$.
\end{proof}

\paragraph{Proof of Theorem~\ref{clique-bound}} We will now prove
Theorem~\ref{clique-bound} using the preceding results.

\begin{proof}
    Write $n = 2^{\nu_2(n)} n_1$. Lemma~\ref{independent-sets} states that we can
    form $\nu_2(n)$ independent sets, discard them, and be left with an
    isomorphic copy of $T(n_1)$. If $n$ is a power of 2, then $n_1 = 1$ and
    every vertex is covered by these independent sets, meaning that $T(n)$ has
    been colored with $\nu_2(n) = 2 \lfloor \log_2 n \rfloor - \nu_2(n)$ colors.

    If $n$ is not a power of 2, then $n_1 > 1$ is odd, and
    Lemma~\ref{missed-classes} states that the congruence classes modulo $4$ are
    independent sets in $T(n)$ except for possibly $0 \bmod 4$ and $n_1 \bmod
    4$. Thus, there are at least two independent sets among the congruence
    classes mod $4$. Choose any two, color them different colors, then discard
    these vertices. If we consider the remaining two congruence classes modulo
    $8$ instead of $4$, they split into four distinct classes.
    Lemma~\ref{missed-classes} implies again that two of these will be
    independent sets, so color those two new colors and refine the remaining two
    classes modulo 16. Repeat this argument modulo 16, 32, 64, and so on, until
    the classes contains only a single vertex of $T(n)$. This process places
    every vertex of $T(n)$ into an independent set, begins at $4 = 2^2$,
    terminates at $2^{\lceil \log_2 n_1 \rceil}$, and produced two independent
    sets at each step. Therefore, $T(n)$ has been colored with
    \begin{equation*}
        2 (\lceil \log_2 n_1 \rceil - 1) + \nu_2(n) = 2 \lfloor \log_2 n \rfloor - \nu_2(n)
    \end{equation*}
    colors.
\end{proof}

\section{Connection to relative primality and the reduced resultant}
\label{sec:equivalence}

Our goal in this section is to give the technical argument which links the
moduli of \cite{zima} to our trinomials. This requires two steps: an argument
relating scalability to the B\'ezout equation, and then an argument relating the
B\'ezout equation to resultants. The arguments of this section rely on a
lesser-known quantity called the \emph{reduced resultant}. We will begin by
gathering some basic definitions and facts about the reduced resultant.

\subsection{Reduced resultants}

For the remainder of this section, we will assume $f, g \in \mathbb{Z}[x]$ are
monic and have unique B\'ezout cofactors $a, b \in \mathbb{Q}[x]$,
meaning that
\begin{equation*}
    a(x) f(x) + b(x) g(x) = 1
\end{equation*}
and $\deg a < \deg g$, $\deg b < \deg f$. When we say ``resultant,'' we mean
``resultant over $\mathbb{Q}$'' unless otherwise specified.

\begin{definition}
    Fix integer coefficient polynomials $f(x)$ and $g(x)$. The smallest positive
    integer obtainable as a $\mathbb{Z}[x]$-linear combination of $f(x)$ and
    $g(x)$ is their \emph{reduced resultant}.
\end{definition}

The set of all integers obtainable as a $\mathbb{Z}[x]$-linear combination of
$f$ and $g$ is an ideal in $\mathbb{Z}$. If $f$ and $g$ are relatively prime,
then this ideal is nontrivial, contains the resultant, and is generated by the
reduced resultant. In general, the reduced resultant does not equal the
resultant.

\paragraph{Example} The resultant of $x^2 - 1$ and $x^2 + 3$ is 16, but their
reduced resultant is 4. To see this, note that
\begin{equation*}
    -\frac{x^2 - 1}{4} + \frac{x^2 + 3}{4} = 1
\end{equation*}
and use the following proposition.

\begin{proposition}
    The reduced resultant of $f$ and $g$ is the least common multiple of the
    reduced denominators of the coefficients of the B\'ezout cofactors $a$
    and $b$.
\end{proposition}

This proof is due to Gerry Myerson \cite{myerson}.

\begin{proof}
    Let $m$ be an integer which can be written as
    \begin{equation}
        \label{Zx-comb}
        A(x) f(x) + B(x) g(x) = m
    \end{equation}
    for some $A, B \in \mathbb{Z}[x]$ with no restrictions on $A$ and $B$. If
    we divide $A$ by $g$ and $B$ by $f$, then we can write
    \begin{align*}
        A &= g q_1 + r_1 \\
        B &= f q_2 + r_2,
    \end{align*}
    where $\deg r_1 < g$ and $\deg r_2 < f$. It follows that
    \begin{equation*}
        r_1 f + r_2 g + fg (q_1 + q_2) = m.
    \end{equation*}
    Because $\deg fg > \deg(r_1 f + r_2 g)$ and the left-hand side equals a
    constant, we must have $q_1 + q_2 = 0$, so
    \begin{equation*}
        r_1 f + r_2 g = m.
    \end{equation*}
    It follows that every integer obtainable as a $\mathbb{Z}[x]$-linear
    combination of $f$ and $g$ can be obtained by one with $\deg A < \deg g$
    and $\deg B < \deg f$. If we divide \eqref{Zx-comb} by $m$, then the
    uniquess of B\'ezout cofactors shows that $A / m$ and $B / m$ \emph{are}
    the B\'ezout cofactors of $f$ and $g$. It follows that the reduced
    resultant is the smallest positive integer $m$ such that $ma$ and $mb$ have
    integer coefficients.
\end{proof}

The most useful property of the reduced resultant, for our purposes, is the
following proposition.

\begin{proposition}
    \label{same-primes}
    If $f(x)$ and $g(x)$ are monic, then their reduced resultant and resultant
    have the same prime divisors.
\end{proposition}

\begin{proof}
    Because the reduced resultant divides the resultant, the primes which divide
    the reduced resultant also divide the resultant. On the other hand, suppose
    that a prime $p$ does \emph{not} divide the reduced resultant. If $d$ is the
    reduced resultant of $f$ and $g$, then we can write
    \begin{equation*}
        A(x) f(x) + B(x) g(x) = d
    \end{equation*}
    for some $A, B \in \mathbb{Z}[x]$. If we reduce this modulo $p$, we obtain
    a nonzero value on the right-hand side, which implies that $f$ and $g$ are
    relatively prime over $\mathbb{F}_p$, and therefore have a nonzero
    resultant over $\mathbb{F}_p$. Because $f$ and $g$ are monic, their
    resultant over $\mathbb{F}_p$ is just their resultant over $\mathbb{Q}$
    reduced modulo $p$, so their resultant over $\mathbb{Q}$ is also not
    divisible by $p$.
\end{proof}

\begin{theorem}
    If the coefficients of the B\'ezout cofactors of $f(x)$ and $g(x)$ are
    dyadic, then $f(x)$ and $g(x)$ dyadically resolve.
\end{theorem}

\begin{proof}
    If the coefficients of the cofactors of $f$ and $g$ are dyadic, then we can
    obtain a power of 2 as a $\mathbb{Z}[x]$-linear combination of $f$ and $g$.
    The reduced resultant is therefore a power of 2, because it divides all
    integers which are $\mathbb{Z}[x]$-linear combinations of $f$ and $g$.
    Proposition~\ref{same-primes} implies that the only prime which could
    divide the resultant of $f$ and $g$ is $2$.
\end{proof}

\subsection{Scalability}

\begin{definition}
    \label{scalable-def}
    The moduli $2^n - 2^k + 1$ and $2^n - 2^j + 1$ are \emph{scalable} if there
    exist polynomials $a(x)$ and $b(x)$ of degree at most $n$ such that:
    \begin{itemize}
        \item $a$ and $b$ have dyadic coefficients;
        \item $a(0)$ and $b(0)$ are integers;
        \item the leading coefficients of $a$ and $b$ are positive; and
    \end{itemize}
    \begin{align*}
        (2^{cn} - 2^{ck} + 1) a(2^c) &\equiv 1 \pmod{2^{cn} - 2^{cj} + 1} \\
        (2^{cn} - 2^{cj} + 1) b(2^c) &\equiv 1 \pmod{2^{cn} - 2^{ck} + 1}
    \end{align*}
    for all sufficiently large $c$.
\end{definition}

\begin{lemma}
    If $f(x)$ and $g(x)$ are polynomials over $\mathbb{Q}$ where
    $f(n)$ and $g(n)$ are integers such that $f(n) \equiv 0 \pmod{g(n)}$ for
    infinitely many integers $n$, then $f(x) \equiv 0 \pmod{g(x)}$.
\end{lemma}

\begin{proof}
    If we Euclidean divide $f(x)$ by $g(x)$, we obtain
    \begin{equation*}
        f(x) = g(x) q(x) + r(x)
    \end{equation*}
    where $\deg r < \deg g$. If $g(n)$ divides $f(n)$ and $f(n) \neq g(n) q(n)$,
    then
    \begin{equation*}
        |r(n)| = |f(n) - g(n) q(n)| \geq \frac{|g(n)|}{\operatorname{denom}(q)}.
    \end{equation*}
    But $|r(n)|$ is strictly smaller than any constant times $|g(n)|$ for
    sufficiently large $n$, because $\deg r < \deg g$. This implies that $f(n) =
    g(n) q(n)$ for all sufficiently large $n$, thus $f(x) = g(x) q(x)$.
\end{proof}

\begin{theorem}
    The moduli $2^n - 2^k + 1$ and $2^n - 2^j + 1$ are scalable if and only if $\tri{x}{n}{k}$ and $\tri{x}{n}{j}$ dyadically resolve.
\end{theorem}

\begin{proof}
    Define the polynomials
    \begin{align*}
        f(x) &= x^n - x^k + 1 \\
        g(x) &= x^n - x^j + 1.
    \end{align*}
    If the moduli are scalable, then there is some inverse
    polynomial $A(x)$ with dyadic coefficients such that $A(2^c) f(2^c) \equiv
    1 \pmod{g(2^c)}$ for all sufficiently large $c$. The previous lemma implies
    that $A(x) f(x) - 1$ is a multiple of $g(x)$, say $A(x) f(x) - B(x) g(x) =
    1$. Further, because $A$, $f$, and $g$ have dyadic coefficients, so does
    $B$. By clearing denominators, we obtain a power of $2$ as a
    $\mathbb{Z}[x]$-linear combination of $f$ and $g$, which implies that their
    reduced resultant is a power of 2. By Proposition~\ref{same-primes}, $f$
    and $g$ dyadically resolve.

    On the other hand, if $f$ and $g$ dyadically resolve, then their reduced
    B\'ezout cofactors $a$ and $b$ in
    \begin{equation*}
        a(x) f(x) + b(x) g(x) = 1
    \end{equation*}
    can be constructed using only dyadic coefficients. These cofactors are
    essentially our inverse polynomials, but they need some minor adjustments to
    match our definition.

    Without loss of generality, we will describe how to make $a$ match the
    definition. If $a(0)$ is not an integer, then the equation $a(0) + b(0) = 1$
    implies $a(0) = c/2^m$ and $b(0) = 1 - c/2^m$ for some integers $c$ and $m >
    1$. Replacing $a$ with $a - gc / 2^m$ and $b$ with $b + fc / 2^m$ makes the
    constant terms integers (at the expense of having degree at most $n$ rather
    than $n - 1$). Now it is true that $a(2^c) f(2^c) \equiv 1 \pmod{g(2^c)}$
    for sufficiently large $c$. If the leading coefficient of $a$ is positive,
    then we are done. If the leading coefficient is negative, then we can
    replace $a$ with $a + c' g$ and $g$ with $g - c' f$ for some sufficiently
    large integer $c'$. Now $a$ satisfies the definition. (Note that the $b$
    which results from this process has a negative leading coefficient. It must
    be repeated in full on $b$ to get a suitable inverse polynomial.)
\end{proof}

\section{Conclusion and open questions}

We have given a simple, resultant-based condition for two moduli of the form
$2^n - 2^k + 1$ to have scalable modular inverses. This condition enables a fast
search for sets of trinomial moduli which are suitable for the Chinese Remainder
Theorem.

There are a number of topics which deserve further study.

\begin{enumerate}
    \item What is the true rate of growth of the clique number $\omega(T(n))$? We suspect that it is asymptotically smaller than $\log n$. A related question is to determine the edge density of $T(n)$, which we suspect is asymptotically smaller than $n^2$.
    \item Is there a formula for the resultant of two trinomials, similar to the
    resultant of two binomials or the resultant of two cyclotomics
    \cite{apostol}? A closely related project is to study sequences of the form
    $c(n) = \res(f(x), x^n - 1)$ for polynomials $f \in \mathbb{Z}[x]$.
    \item We only use ``$2$'' as a base for our moduli because we work with
    binary computers. Many of our results carry over verbatim if $2$ is replaced
    with an arbitrary prime $p$. Studying ``$p$-adically resolving'' trinomials,
    the $p$-adic valuations of these resultants, and so on, seems interesting in
    its own right.
\end{enumerate}

\paragraph{Acknowledgements} We thank the Digital Research Alliance of Canada
and the Office of Research Computing and Data Services at Dartmouth College for
computational resources to run experiments.

\printbibliography

\end{document}